\documentclass[11pt,letterpaper]{amsart}
\usepackage{geometry}
\usepackage{mathrsfs,amsmath,amssymb,amsthm,amsfonts}
\usepackage{latexsym,amscd,setspace,color}
\usepackage{mathpazo} 
\usepackage[active]{srcltx}
\usepackage[mathscr]{eucal}
\newtheorem{theorem}{Theorem}
\newtheorem{proposition}{Proposition}
\newtheorem{lemma}{Lemma}
\newtheorem{corollary}{Corollary}
\theoremstyle{remark}
\newtheorem{remark}{Remark}

\newcommand{\C}{\mathbb{C}}
\newcommand{\R}{\mathbb{R}}
\newcommand{\N}{\mathbb{N}}
\newcommand{\Z}{\mathbb{Z}}

\newcommand{\D}{\Omega}

\newcommand{\dbar}{\overline{\partial}}
\newcommand{\rem}{\mathscr R}
\DeclareMathOperator{\re}{Re}
\DeclareMathOperator{\im}{Im}
\DeclareMathOperator{\Arg}{Arg}
\newcommand{\F}{\mathcal{F}}
\usepackage{hyperref}
\onehalfspace

\title[High Dimensional Worm Domains]{Irregularity of the Bergman projection on 
\\ Worm Domains in $\C^{n}$}

\author{David Barrett}
\address[David Barrett]{University of Michigan, Department of Mathematics, Ann
Arbor, MI 48109, USA}
\email{barrett@umich.edu}

\author{S\"{o}nmez \c{S}ahuto\u{g}lu}
\address[S\"{o}nmez \c{S}ahuto\u{g}lu]{University of Toledo, Department of
Mathematics, Toledo, OH 43606, USA}
\email{sonmez.sahutoglu@utoledo.edu}

\thanks{The first author is supported in part by NSF grant number DMS-0901205.
The second author is  supported in part by NSF grant number DMS-0602191.}

\date{\today}
\begin{document}

\begin{abstract}
We construct higher-dimensional versions of the Diederich-Forn\ae{}ss worm
domains and show that the Bergman projection operators for these domains are not
bounded on high-order $L^p$-Sobolev spaces for $1\leq p<\infty.$
\end{abstract}

\maketitle

\section{Introduction} \label{Intro}

Let $\D$ be a bounded domain in $\C^n$ and $A^{2}(\D)$ denote the Bergman space
of square-integrable holomorphic functions on $\D.$ The Bergman projection on 
$\D$ is the orthogonal projection from $L^2(\D)$ onto $A^{2}(\D).$

The Bergman projection is known to be regular, in the sense that it maps $W^s$ to
$W^s$ for all $s\geq 0$ where $W^s$ denotes the Sobolev space of order $s,$ on a
large class of smooth bounded pseudoconvex domains (throughout this paper a domain is
smooth if its boundary is a smooth manifold). Regularity is, usually, 
established through the $\dbar$-Neumann problem, the solution operator for the
complex Laplacian $\Box=\dbar\dbar^*+\dbar^*\dbar$ on square integrable
$(0,1)$-forms. For more information on this matter we refer the  reader to
\cite{BoasStraube99,StraubeBook} and the references therein.

Irregularity of the Bergman projection is not understood nearly as well as 
regularity. The story of irregularity goes back to the discovery of the
worm  domains in $\C^2$ by Diederich and Forn\ae{}ss \cite{DiederichFornaess77}. 
Worm domains were constructed to show that  the closure of some smooth
bounded pseudoconvex domains may not have Stein neighborhood bases (a compact
set $K\subset \C^n$ is said to have a Stein neighborhood basis if for every open
set $U$ containing  $K$ there exists a pseudoconvex domain $V$ such that 
$K\subset V \subset U$). Indeed, Diederich and Forn\ae{}ss in
\cite{DiederichFornaess77} showed that the closure a worm domain does not
have a Stein neighborhood basis if the total winding is bigger than or equal to
$\pi.$ It turned out that worm domains are also counter-examples for regularity
of the Bergman projection. In 1991,  Kiselman \cite{Kiselman91} showed that the
Bergman projection does not satisfy Bell's condition R on nonsmooth  worm
domains (a domain $\D$ satisfies Bell's condition 
R if  the Bergman projection maps $C^\infty(\overline\Omega)$ to
$C^\infty(\overline\Omega)$). In 1992,  the first author 
\cite{Barrett92} showed that the Bergman projection on a smooth worm domain   does not map
$W^s$ into $W^s$ if $s\geq \pi /(\text{total winding}).$  On the other hand, Boas
and Straube \cite{BoasStraube92} showed that the Bergman projection maps $W^k$
into $W^k$ if $k\leq \pi /(2\times \text{total winding})$ and $k$ is a positive
integer or $k=1/2.$ Finally, in 1996 Christ \cite{Christ96} showed that the
Bergman projections on smooth worm domains, with any positive
winding, do not satisfy Bell's condition R. Recently, Krantz and Peloso
\cite{KrantzPeloso08a,KrantzPeloso08b} studied the asymptotics
for the Bergman kernel on the model domains in $\C^2$ and derived $L^p$
(ir)regularity for the Bergman projection on worm domains in $\C^2.$ 

In this note we will construct smooth bounded pseudoconvex domains 
$\D_{\alpha\beta}\subset \C^n$ that are  higher dimensional generalizations of
the worm domains in $\C^2$ and  study the irregularity of the Bergman projection
on  these domains on $L^p$ Sobolev spaces for $1\leq p<\infty.$ We will use 
 the method developed by the first author in \cite{Barrett92} to show that
irregularity on $L^2$ Sobolev spaces depends only on the total winding whereas the
irregularity on $L^p$ spaces with $p\neq 2$ depend on the total winding as well as the
dimension $n$.

The two parameters $\alpha$ and $\beta$ in $\D_{\alpha\beta}$ represent the speed of
the winding and the thickness of the annulus, respectively. Both parameters play a role in
the proof of Theorem \ref{Theorem}, but we find it interesting to note that the actual
results depend only on the total winding whether this is achieved by fast winding along a
thin annulus or slow winding along a thick annulus. 

The domains  $\D_{\alpha\beta} \subset \C^n, n\geq 3,$ are defined by     
\begin{equation*}\label{Domain}
\D_{\alpha\beta}=\left\{(z_{1}, z',z_{n})\in \C^{n}: r(z_1,z',z_n) <0\right\}
\end{equation*}
with
\[r(z_1,z',z_n) =\left|z_{1}-e^{2i\alpha \ln|z_{n}|}
\right|^{2}+|z'|^{2}-1 +\sigma(|z_n|^2-\beta^2)+ \sigma(1-|z_n|^2);\]
here  $z'=(z_{2},\ldots,z_{n-1}), |z'|^{2}=|z_{2}|^{2}+\cdots+|z_{n-1}|^{2}$, 
the constants $\alpha>0, \beta>1,$  and
$\sigma(t)=Me^{-1/t}$ for $t> 0,\,\sigma(t)=0$ for $t\leq 0$ for some $M>0$.

In section \ref{levi} below we show that $\D_{\alpha\beta}$ 
is smooth bounded pseudoconvex when $M$ is sufficiently large.
The main result of this paper is the following theorem.

\begin{theorem}\label{Theorem}  
The Bergman projection for $\D_{\alpha\beta}$ does not map
$W^{p,s}\left(\D_{\alpha\beta}\right)$ into
$W^{p,s}\left(\D_{\alpha\beta}\right)$ where $1\leq p<\infty$ and  
$s\ge \frac{\pi}{2\alpha\ln\beta}+n\left(\frac{1}{p}-\frac{1}{2}\right).$
\end{theorem}
Here $W^{p,s}\left(\D_{\alpha\beta}\right)$ is the Sobolev space of order
$s$ with exponent $p$ and when 
$W^{p,s}\left(\D_{\alpha\beta}\right)\not\subset L^2\left(\D_{\alpha\beta}\right)$ 
we mean that the $W^{p,s}$ bounds do not hold for the Bergman projection on 
$W^{p,s}\left(\D_{\alpha\beta}\right)\cap L^2\left(\D_{\alpha\beta}\right).$ The
denominator $2\alpha\ln\beta$ appearing above may be interpreted as the 
total amount of winding along the annulus $1<|z_{n}|<\beta$ 
(see \eqref{LimitDomain} below).  

If we choose $p=2$ then  the amount of irregularity provided by a fixed amount
of winding is independent of the dimension. 

\begin{corollary}\label{corollary1}  
The Bergman projection for $\D_{\alpha\beta}$ does not map
$W^{2,s}\left(\D_{\alpha\beta}\right)$ to
$W^{2,s}\left(\D_{\alpha\beta}\right)$ when $s\geq 
\frac{\pi}{2\alpha\ln\beta}$.
\end{corollary}

\begin{remark} 
Assume that the Bergman projection $P_U$ of a domain $U$ bounded on $L^p(U)$
where $p> 2.$ Then the duality and self-adjointness of the Bergman
projection imply that $P_U$ is also bounded on $L^q(U)$
where $\frac{1}{p}+\frac{1}{q}=1.$ Furthermore, interpolation implies that
$P_U$ is bounded on $L^r$ for all $r\in [q,p].$ 
\end{remark}
Therefore, when $s=0$ and $n\alpha\ln\beta>\pi,$ the previous remark and Theorem
\ref{Theorem}  imply the following corollary.
 
\begin{corollary}\label{corollary2}  
The Bergman projection for $\D_{\alpha\beta}$ does not map
$L^p\left(\D_{\alpha\beta}\right)$ to
$L^p\left(\D_{\alpha\beta}\right)$ when
$0<\frac{1}{p}\le\frac{1}{2}-\frac{\pi}{2n\alpha\ln\beta}$ or
$\frac{1}{2}+\frac{\pi}{2n\alpha\ln\beta}\leq \frac{1}{p}<1$.
\end{corollary}

Theorem \ref{Theorem} is proved in section \ref{ProveThm} below.  The proof is
based on model domain asymptotics developed in section \ref{Model}.

\section{Geometry of the Worm Domains} \label{levi}

\begin{proposition}
The domain $\D_{\alpha\beta}$ is smooth bounded and pseudoconvex whenever
$M$ is sufficiently large.
\end{proposition}

\begin{proof} 
Start by requiring $M>e^2$.  Then  $\Omega \subset \{z\in \C^n:
|z_1|<3,|z'|<2,1/2 <|z_n|<\sqrt{\beta^2+1/2}\}$. Then $\Omega$ is bounded.
Also, by considering $z_1$-, $z'$-, and $z_n$-derivatives in order it is easy to
check that the gradient of $r(z)$ does not vanish on $\{z\in \C^n:r(z)=0\}$, so
$\Omega$ has smooth boundary.

It remains to show that $\D_{\alpha\beta}$ is pseudoconvex. It suffices to
check this locally.  We focus on the case $|z_n|\geq (1+\beta)/2$, the case
$|z_n|\leq (1+\beta)/2$ being similar. 

Multiplying $r(z)$ by $e^{\Arg(z^{2\alpha}_{n})}$
we obtain the new defining function
\begin{equation*}
r_1(z)=r_2(z)-2\,\re\left(z_1 z_n^{-2\alpha i}\right)
\end{equation*}
where
\[r_2(z) = \left( |z_1|^2 + |z'|^2+\lambda(z_n)\right) e^{\Arg(z^{2\alpha}_{n})}
\text{ and } \lambda(z_n)=\sigma\left(|z_n|^2-\beta^2\right).\]
Since $2\,\re\left(z_1 z_n^{-2\alpha i}\right)$ is pluriharmonic it will
suffice now to show that $r_2$ is plurisubharmonic.
To simplify the notation let $A(z)=|z_1|^2+|z'|^2+\lambda(z_n)$ and
$B(z)=\Arg(z^{2\alpha}_{n}).$ Let $W=\sum_{j=1}^nw_j \partial/\partial z_j$ with
$w_j$ constant. In the following calculations  $H_{f}(W)$ denote the complex
Hessian of $f$ in the direction $W.$ Then $W(r_2)=e^B(W(A)+AW(B))$ and
Cauchy-Schwarz inequality implies that 
\[-2\re\left( \overline w_n B_{\overline 
z_n}\sum_{j=1}^{n-1} w_j\overline z_j \right)\leq
\sum_{j=1}^{n-1}|w_j|^2+|\overline w_nB_{\overline 
z_n}|^2 \sum_{j=1}^{n-1}|z_j|^2.\]
Using the above inequality in the second line below we get
\begin{align*}
H_{r_2}(W)
&=e^B(H_A(W)+2\re(W(A)\overline{W}(B))+A|W(B)|^2+A H_B(W)) \\
&\geq |w_n|^2e^B(\lambda_{z_n\overline  z_n}+2\re(\lambda_{z_n}B_{\overline 
z_n})+\lambda |B_{\overline  z_n}|^2).
\end{align*}
One can check that  
$\lambda_{z_n}(z_n)=\overline{z}_n \sigma'(|z_n|^2-\beta^2),|B_{\overline
z_n}|=\frac{\alpha}{|z_n|},$ and 
\begin{equation*}
\lambda_{z_n\overline{z}_n}(z_{n})= |z_n|^2\sigma''(|z_n|^2-\beta^2)
+\sigma'(|z_n|^2-\beta^2).
\end{equation*}
We note that since $\lambda
(z_n)=\lambda_{z_n}(z_n)=\lambda_{z_n\overline z_n}(z_n)=0$ for 
$|z_n|\leq \beta,$ without loss of generality  we can assume that
$|z_n|>\beta.$ Using  the fact that $\beta<|z_n|<\sqrt{\beta^2+1/2}$ and
$t=|z_n|^2-\beta^2$ on the third line below we get 
\begin{align*}
\lambda_{z_n\overline z_n}
+2\re(\lambda_{z_n}B_{\overline  z_n})+\lambda |B_{\overline  z_n}|^2 \geq & 
\lambda_{z_n\overline z_n}
-\frac{2\alpha|\lambda_{z_n}|}{|z_n|}\\
\geq &|z_n|^2\sigma''(|z_n|^2-\beta^2)
+(1-2\alpha)\sigma'(|z_n|^2-\beta^2) \\
=& Me^{-1/t} \left(
\frac{\beta^2+t}{t^4}-\frac{2(\beta^2+t)}{t^3}+\frac{1-2\alpha}{t^2}
\right)\\
=&\frac{M(\beta^2+t)e^{-1/t}}{t^4} \Big( 1-2t+\frac{(1-2\alpha)t^2}{\beta^2+t}
\Big)
\end{align*}
We can choose $M$ sufficiently large so that $z\in \D_{\alpha\beta}\cap
\{z\in\C^n:|z_n|\geq \beta\}$  implies that $t$ is sufficiently small. In return, this
implies that 
\[1-2 t+\frac{(1-2\alpha)t^2}{\beta^2+t} > 0.\]
The last inequality above implies that 
$\lambda_{z_n\overline z_n} +2\re(\lambda_{z_n}B_{\overline  z_n})+\lambda
|B_{\overline  z_n}|^2 \geq 0$ for 
$z\in \D_{\alpha\beta}$ such that $|z_n|\geq (1+\beta)/2.$
Hence, the domain $\D_{\alpha\beta}$ is pseudoconvex for sufficiently large $M.$
\end{proof}


\begin{remark}
A similar calculation shows that the set of weakly pseudoconvex points in the
boundary is the set $\{(0,\ldots,0,z_n)\in \C^n:1\leq |z_n| \leq \beta \}.$
\end{remark}

\begin{remark}
We note that regularity of the $\dbar$-Neumann operator is closely connected
to regularity of the Bergman projection \cite{BoasStraube90}. In particular, 
if the $\dbar$-Neumann operator of a smooth bounded pseudoconvex domain is
globally regular then  the Bergman projection satisfies Bell's condition R.  
One can show that on the set $\{(0,\ldots,0,z_{n})\in \C^n:1\leq |z_n| \leq
\beta \}$ the Levi form of $r$ has only one vanishing eigenvalue as the Levi
form has positive eigenvalues in the direction transversal to $z_{n}$-axis. In
this case Theorem 1 in \cite{SahutogluStraube06} applies and it implies that the
$\dbar$-Neumann operator is not compact on $(0,1)$-forms (compactness of the
$\dbar$-Neumann operator implies that it is globally regular
\cite{KohnNirenberg65}). However, to show irregularity of the Bergman
projection in Sobolev scale one needs to work harder. 
\end{remark}

\section{Model Domains} \label{Model}

In this section we are going to define a family of simplified model domains and
calculate the asymptotics for the Bergman kernels of these model domains. We
use a modified version of the method developed by the first author
in \cite{Barrett92}. 

For $\lambda>0$ let
\begin{align*}
\tau_{\lambda}(z_{1},z',z_{n})&=(2\lambda^{2}z_{1},\lambda z',z_{n}),\\
r_{\lambda}&=\lambda^{2} r \circ \tau_{\lambda}^{-1},\\
D_{\lambda}&=\tau_{\lambda}(\D_{\alpha\beta}).
\end{align*}
Then for $1\le|z_{n}|\le\beta$ we have $r_{\lambda}\searrow r_{\infty}$ as
$\lambda\to \infty$ where 
\[r_{\infty}(z_{1},z',z_{n})=|z'|^{2}-\re\left(z_{1}e^{
-2\alpha i\ln|z_{n}|}\right);\] 
for $|z_n|$ outside this range we have $r_{\lambda}\to\infty$.  It follows that
the $D_{\lambda}$ converge in an appropriate sense to the limit domain
\begin{equation}\label{LimitDomain}
D=D_{\alpha\beta}= \left\{(z_{1},z',z_{n})\in \C^{n}:\re\left(
z_{1}e^{-2\alpha
i\ln|z_{n}|}\right) >|z'|^{2},1<|z_{n}|<\beta\right\},
\end{equation}
the limit being increasing over the annulus $1\le|z_{n}|\le\beta$.

Bergman projection $P$ of $D$ is defined by $Pf(z)=\int_D K(z,w) f(w)\,dV(w)$
where $f\in L^{2}(D)$ and  $K:D\times D\to \C,$ is the Bergman kernel
characterized by the following conditions
\begin{itemize} 
 \item [i.] $K(z,w)\in A^2(D)$ for fixed $w\in D$,
\item[ii.] $K(w,z)=\overline{K(z,w)}$,
\item[iii.] $\int_D K(z,w) f(w)\,dV(w)=f(z)$ for $f\in A^{2}(D)$.
\end{itemize}
If $f_1, f_2,\dots$ is an orthonormal basis for $A^{2}(D)$ then we
have $K(z,w)=\sum_j f_j(z)\overline{f_j(w)}.$  

To study the Bergman kernel of $D$ we begin by performing a Fourier
decomposition. We define   
\begin{equation}\label{projection}
(P_{Jk}f)(z_{1},z',z_{n})=\frac{1}{2^{n-1}\pi^{n-1}}\int\limits_{[-\pi,\pi]^{n-1
}}f(z_{1},e^{iS}z',e^{it}z_{n})e^{-iJ S}e^{-ikt}dS\,dt,
\end{equation}
where 
\begin{align*}
e^{iS}&=(e^{is_{1}},\ldots,e^{is_{n-2}}),\\
S&=(s_{1},\ldots,s_{n-2})\in [-\pi,\pi]^{n-2},\\
J&=(j_{1},\ldots,j_{n-2})\in \N^{n-2},\\
k&\in \Z,\\
J S&=j_{1}s_{1}+\cdots+j_{n-2}s_{n-2},\\
dS&=ds_{1}\cdots ds_{n-2}.
\end{align*}
Let us define the mapping 
$\rho_{St}(z_{1},z',z_{n})=(z_{1},e^{iS}z',e^{it}z_{n}).$
Then $P_{Jk}$ is the orthogonal projection from $A^{2}(D)$ onto 
\[A_{Jk}^{2}(D)=\{f\in A^{2}(D): f\circ \rho_{St}=e^{iJS}e^{ikt}f \text{ for all
} S,t\}.\]
Therefore the Bergman space $A^{2}(D)$  can be written as an orthogonal sum 
\[A^{2}(D)=\underset{\begin{subarray}{c}
 J\in \N^{n-2},\,k\in\Z \end{subarray}}{\bigoplus} A_{Jk}^{2}(D)\]
and the Bergman kernel $K(z,w)$ for $D$ satisfies
\[K(z,w)=\sum\limits_{J\in \N^{n-2},\,k\in\Z}K_{Jk}(z,w)\] 
where $K_{Jk}(z,w)$ is the kernel for $A_{Jk}^{2}(D).$  

One can show that for $f\in A_{Jk}^{2}(D)$ the function
$f(z_{1},z',z_{n})z_{2}^{-j_1}\cdots z_{n-1}^{-j_{n-2}}z_{n}^{-k}$ is  a
function that is locally independent of $(z',z_n)$.  We notate such functions as
functions of $z_1$, where it is understood that $z_1$ ranges over the Riemann
domain described by $-\pi/2<\Arg z_1<2\alpha\ln\beta+\pi/2$.

Let $|J|=j_{1}+\cdots+j_{n-2}.$  Then a square integrable holomorphic function
$f$ on $D$ can be written as   
\[ f(z)=\sum_{J\in \N^{n-2},\,k\in \Z }F_{Jk}(z)\,\]
where
\[F_{Jk}(z_{1},z',z_{n})=z_{1}^{-\frac{|J|+n}{2}}f_{Jk}(z_{1})z'^{J}z_{n}^{k}\]
and the sum converges locally uniformly. 

Now we will calculate the $L^{2}$-norm of $F_{Jk}$ on $D.$ Let
$z_{1}=r_{1}e^{i\theta_{1}}, r_{j}=|z_{j}|$ for $j=1,\ldots n
,r'=\sqrt{r_{2}^{2}+\cdots+r_{n-1}^{2}},s=\ln|z_{n}|^{2}.$ Then $D$ is described
by the inequalities 
\begin{align*}
0&< r_{1}<\infty,\\
0&<s<2\ln\beta,\\
 |\theta_{1}-\alpha s|&<\pi/2,\\
 0&\leq r'<\sqrt{r_{1}\cos(\theta_{1}-\alpha s)}.
\end{align*}
We have 
\begin{eqnarray} \displaystyle 
\nonumber  \|F_{Jk}\|^{2}_{D}&=&
\int\limits_{D}
 |f_{Jk}(r_{1}e^{i\theta_{1}})|^{2}r_{1}^{-|J|-n+1}r_2^{2j_2+1}\cdots
r_{n-1}^{2j_{n-2}+1} r_{n}^{2k+1}d\theta_{1}\cdots d\theta_{n} dr_{1}\cdots
dr_{n} \\
 \nonumber &=&
 C_{nJ} \int\limits_{\substack{0< r_{1}<\infty \\
 |\theta_{1}-\alpha s|<\pi/2 \\
  0<s<2\ln \beta}}
   |f_{Jk}(r_{1}e^{i\theta_{1}})|^{2}\cos^{|J|+n-2}(\theta_{1}-\alpha
s)e^{s(k+1)}r_{1}^{-1} d\theta_{1} dr_{1} ds  \\
 &=& \label{IntOnGamma}
\int\limits_{\substack{
 0< |z_{1}|<\infty \\
 -\pi /2<\arg(z_{1})<2\alpha \ln \beta +\pi /2 }}
 |f_{Jk}(z_{1})|^{2}W_{Jk}(\theta_{1})|z_{1}|^{-2}\,dV(z_{1})
 \end{eqnarray}
 where  $C_{nJ}$ is a positive constant,
\[W_{Jk}(\theta_{1})=C_{nJ}\int_{-\infty}^{\infty}\cos^{|J|+n-2}(\theta_{1}
-\alpha t)\chi_{\pi/2}(\theta_{1}-\alpha t)e^{t(k+1)}\chi_{\ln
\beta}(t-\ln\beta)\,dt, \]
and $\chi_a(t)$ is the characteristic function of the interval $[-a,a]$ for
$a>0.$  (The positivity of   $C_{nJ}$ follows from the fact that we are only
integrating over positive values of $r_j$.)

Let us use a change of coordinates $z=\ln z_1$ in the last integral to obtain 
\begin{align}\label{ChangeCoord}
\|F_{Jk}\|^{2}_{D}&= \int\limits_{\substack{
 -\infty< x<\infty \notag\\
 -\pi /2<y<2\alpha \ln \beta +\pi /2 }}|f_{Jk}(e^z)|^{2}W_{Jk}(y)\,dV(z) \\
 &= \int\limits_{\substack{-\infty< x<\infty \\
 -\pi /2<y<2\alpha \ln \beta +\pi /2 }}|\widetilde
f_{Jk}(z)|^{2}W_{Jk}(y)\,dV(z)
 \end{align}
where $z=x+iy$ and  $\widetilde f_{Jk}(z)=f_{Jk}(e^{z}).$  
Then $\widetilde f_{Jk}$ is a square integrable holomorphic function on
$S_{\alpha\beta}=\{z\in \C:-\pi/2<\im(z)<\pi /2+2\alpha \ln \beta \}$ with
weight $W_{Jk}.$ Furthermore, the Bergman kernel $K_{Jk}$ for $A_{Jk}^{2}(D)$
can be calculated as 
\begin{equation}\label{EqnKernelTransfrom}
K_{Jk}(z,w)=K_{Jk}^{\alpha\beta} (\ln z_{1},\ln
w_{1})\frac{z'^{J}z^{k}_{n}\overline  w'^{J}\overline 
w^{k}_{n}}{z_{1}^{\frac{|J|+n}{2}}\overline  w^{\frac{|J|+n}{2}}_{1}}
\end{equation}
where $K_{Jk}^{\alpha\beta}$ is the Bergman kernel on $S_{\alpha\beta}$ with the
weight $W_{Jk}.$ 
(One way to see this is to note that \eqref{ChangeCoord} allows us to convert an
orthonormal basis for the Bergman space on $S_{\alpha\beta}$ with  weight
$W_{Jk}$ to an orthonormal basis for $A_{Jk}^{2}$.)

Let $\F(f)$ denote the Fourier transform of $f$; thus
$\F(f)(\xi)=\frac{1}{\sqrt{2\pi}}\int_{-\infty}^{\infty}f(t)e^{-i\xi t}dt$ and
$\F^{-1}(f)(x)=\frac{1}{\sqrt{2\pi}}\int_{-\infty}^{\infty}f(\xi)e^{i\xi
t}d\xi$.

\begin{proposition}\label{JkKernel} 
$K^{\alpha\beta}_{Jk}$ is given by the integral
\begin{equation}\label{IntFormula}
K^{\alpha\beta}_{Jk}(z,w)=\frac{1}{\sqrt{2\pi}}\int_{\R} \frac{ e^{i(z-\overline
 w)\xi}}{\F(W_{Jk})(-2i\xi)}d\xi.
\end{equation}
\end{proposition}

\begin{proof}
See \cite{Barrett92} and \cite[Lemma 6.5.1]{ShawBook}.
\end{proof}

Note also that $-\pi<\text{Im}(z-\overline  w)<\pi+4\alpha\ln\beta $ for $
z,w\in S_{\alpha\beta}.$

\begin{proposition}\label{WeightTransform}
The Fourier transform of $W_{Jk}$ is given by 
\begin{equation}\label{WeightTransformFormula}
\F(W_{Jk})(\xi)
=D_{nJ}e^{-\frac{i\xi\pi}{2}}\frac{E_{Jk}(\xi)}{
(\xi+|J|+n-2)(\xi+|J|+n-4)\ldots(\xi-|J|-n+2)}  
\end{equation}
 where 
\[E_{Jk}(\xi)=\left(e^{i\xi\pi}-(-1)^{|J|+n}\right)\left(\frac{e^{
2(k+1-i\alpha\xi)\ln\beta}-1}{k+1-i\alpha\xi}\right). \]
\end{proposition} 

We postpone the proof of this Proposition. 

To apply residue methods to \eqref{IntFormula} we need to find the zeros of
$\F(W_{Jk})(-2i\xi)$. Let us denote the set $\{s\in \mathbb{Z}:-m\leq s\leq m\}$
by $\mathbb{I}(m).$ From Proposition \ref{WeightTransform} we see that if
$|J|+n$ is even  then the zeros of $\F(W_{Jk})(-2i\xi)$ are located at
\[\left\{ mi:m\in
\mathbb{Z}\setminus\mathbb{I}\left(\frac{|J|+n-2}{2}\right)\right\} \bigcup
\left\{\frac{m\pi i}{2\alpha\ln\beta}+\frac{k+1}{2\alpha
}:m\in \mathbb{Z}\setminus\{0\}\right\} \] 
and in case  $|J|+n$ is odd they are located at
\begin{multline*}
\left\{ mi+\frac{ i}{2}:m\in
\mathbb{Z}\setminus\left(\mathbb{I}\left(\frac{|J|+n-3}{2}\right)\cup\{
-(|J|+n-1)/2\}\right)\right\} \\
\bigcup \left\{\frac{m\pi i}{2\alpha\ln\beta}+\frac{k+1}{2\alpha
}:m\in \mathbb{Z}\setminus\{0\}\right\}. 
\end{multline*} 

 For simplicity we focus now on the case  $J=0,\, k=-2$; note that this
guarantees that the zeros enumerated above are simple  (see Remark \ref{kChoice}
below).
 
Let $\nu_{\alpha\beta}=\frac{\pi}{2\alpha\ln\beta}$ and
$\mu_{\alpha}=\frac{1}{2\alpha }>0$.

\begin{proposition} \label{Kernel}
The  kernels $K_{0,-2}$ satisfy
\begin{multline}\label{BergmanKernel}
K_{0,-2}(z,w)= \sum_{\ell=0}^{[\nu_{\alpha\beta}-n/2]}
C_\ell z_{1}^{\ell}\overline w_{1}^{-\ell-n}
z_n^{-2}\overline  w_n^{-2}
+Cz_{1}^{\nu_{\alpha\beta}-n/2-i\mu_{\alpha} }\overline 
w_1^{-\nu_{\alpha\beta}-n/2+i\mu_{\alpha}} z_n^{-2}\overline  w_n^{-2}
+\rem(z,w)
\end{multline}  
 where $\varepsilon>0$,  the constants $C$ and $C_\ell$ are nonzero  and  the
remainder term $\rem(z,w)$ satisfies
 \begin{equation*}
\left(\frac{\partial}{\partial z_{1}}\right)^{m}\rem(z,w) =
O\left(z_1^{\nu_{\alpha\beta}-n/2+\varepsilon-m}
\overline w_1^{-\nu_{\alpha\beta}-n/2-\varepsilon}\right)
\end{equation*}
uniformly on closed subannuli of $1<|z_{n}|<\beta$.
\end{proposition}

\begin{proof}  
We apply the residue theorem to the integral in  \eqref{IntFormula} along the
strip $-\nu_{\alpha\beta}-\varepsilon\le \im\xi\le 0$ to obtain
\begin{equation*}
K^{\alpha\beta}_{0,-2}(z,w) =
\sum_{\ell=0}^{[\nu_{\alpha\beta}-n/2]} C_\ell e^{(\ell+\frac{n}{2})( z-
\overline w)} + C e^{(\nu_{\alpha\beta}-i\mu_{\alpha})( z- \overline w)} +
\widetilde\rem(z,w)
\end{equation*}
for non-zero $C, C_\ell$, where $\widetilde\rem(z,w)$ and all of its derivatives
are $O\left( e^{(\nu_{\alpha\beta}+\varepsilon)( z-\overline w)}\right)$
on closed substrips of $S_{\alpha\beta}$. 

Plugging this into \eqref{EqnKernelTransfrom} we obtain \eqref{BergmanKernel}. 
\end{proof}

\begin{remark}\label{kChoice}
We have focused on the case $J=0,\, k=-2$   because this is the simplest choice
which avoids possible problems with double poles.  Analogous formulae hold for
other values of $k$ in the  absence of double poles.  When double poles do occur
they contribute factors of $\ln(z_1-\overline w_1)$.
\end{remark}

\begin{lemma}\label{SumEval}
$\displaystyle{ \sum_{s=0}^{j}\binom{j}{s} \frac{(-1)^{s}}{\xi+\alpha(j-2s)}
= \frac{(-2\alpha)^j j!} {(\xi+\alpha j)(\xi+\alpha(j-2))\cdots(\xi-\alpha
j)}}.$
\end{lemma}
\begin{proof}
The statement is true for $j=0$.

Working inductively and recalling that
$\binom{j}{s}=\binom{j-1}{s-1}+\binom{j-1}{s}$ we have
\begin{align*}
\sum_{s=0}^{j}\binom{j}{s}\frac{(-1)^{s}}{\xi+\alpha(j-2s)}
&=
\sum_{s=0}^{j-1}\binom{j-1}{s}\frac{(-1)^{s}}{\xi+\alpha(j-2s)}
+
\sum_{s=1}^{j}\binom{j-1}{s-1}\frac{(-1)^{s}}{\xi+\alpha(j-2s)}
\\
&=
\frac{(-2\alpha)^{j-1} (j-1)!}
{(\xi+\alpha j)(\xi+\alpha(j-2))\cdots(\xi+\alpha(-j+2))}\\
&\qquad -\frac{(-2\alpha)^{j-1} (j-1)!}
{(\xi+\alpha(j-2))(\xi+\alpha(j-4))\cdots(\xi-\alpha j)}\\
&=
\frac{(-2\alpha)^{j-1} (j-1)!}
{(\xi+\alpha(j-2))\cdots(\xi+\alpha(-j+2))}
\left(
\frac{1}{\xi+\alpha j}-\frac{1}{\xi-\alpha j}
\right)\\
&=\frac{(-2\alpha)^j j!}
{(\xi+\alpha j)(\xi+\alpha(j-2))\cdots(\xi-\alpha j)}.
\end{align*}
\end{proof}
\begin{proof}[Proof of Proposition \ref{WeightTransform}]
Write 
\[W_{Jk}(y)= C_{nJ} \Big( W_{Jk1} * W_{Jk2}\Big) (y/\alpha)\]
for $ -\pi /2<y<\pi/2+2\alpha \ln \beta$ where $f*g$ denotes the convolution of
$f$ and $g$ and  
\begin{align*}
W_{Jk1}(t)&=\cos^{|J|+n-2}(\alpha t)\chi_{\pi/2}(\alpha t), \\
 W_{Jk2}(t)&=e^{t(k+1)}\chi_{\ln \beta}(t-\ln\beta).
\end{align*}
To calculate the Fourier transform of $W_{Jk}$ we first calculate 
\[\cos^{j}(t)=\frac{1}{2^{j}}\sum_{s=0}^{j}\binom{j}{s}e^{i(2s-j)t}.\]
One can calculate that 
\[\F(\cos^{j}(t)\chi_{\pi/2}(t))(\xi)=\frac{1}{i\sqrt{2\pi}2^{j-1}}
\sum_{s=0}^{j}\binom{j}{s}\frac{\left(e^{\frac{i(\xi+j-2s)\pi}{2}}-e^{-\frac{i
(\xi+j-2s)\pi}{ 2 }}\right)}{2(\xi+j-2s)}.\]
 Lemma \ref{SumEval} implies that
\begin{align*}
\F(\cos^{j}(\alpha t)\chi_{\pi /2}(\alpha t))(\xi) 
&= \frac{1}{\alpha}\F(\cos^{j}( t)\chi_{\pi /2}(t))(\xi/\alpha)
\\ 
&=\frac{i^{j-1}\Big(e^{\frac{i\xi\pi}{2\alpha}}-(-1)^{j}e^{-\frac{i\xi\pi}{
2\alpha}}\Big)}{\sqrt{2\pi}2^{j}}\sum_{s=0}^{j}\binom{j}{s}\frac{(-1)^{s}}{
\xi+\alpha (j-2s)}\\
&=\frac{(-\alpha i)^j j!
\Big(e^{\frac{i\xi\pi}{2\alpha}}-(-1)^{j}e^{-\frac{i\xi\pi}{2\alpha}}\Big)}
{i\sqrt{2\pi}{(\xi+\alpha j)(\xi+\alpha(j-2))\cdots(\xi-\alpha j)}}.
\end{align*}

We also need to find the Fourier transform of $e^{kt}\chi_{a}(t-a)$:
 \[\F(e^{kt}\chi_{a}(t-a))(\xi)= \frac{1}{\sqrt{2\pi}}\frac{e^{2a(k-i\xi)}-1}{
k-i\xi}.\]
Using $\F(f*g)=\sqrt{2\pi}\F(f)\F(g)$ we find that the Fourier transform of
$W_{Jk}$ is given by \eqref{WeightTransformFormula}.
\end{proof}

\section{Proof of Theorem \ref{Theorem}}\label{ProveThm}

The proof of Theorem \ref{Theorem} follows immediately from
Lemmas \ref{SupposeContrary}  and \ref{Bad-f} below.

\begin{lemma}
If $P$ is continuous on $W^{p,s}(\D_{\alpha\beta})$ then 
\begin{equation}\label{Estimate}
\left\| |r_{\lambda}|^{t} \left(\frac{\partial}{\partial z_1}\right)^m
P_{\lambda}f\right\|_{L^{p}(D_{\lambda})} \leq C  \left\|
f\right\|_{W^{p,s}(D_{\lambda})}
\end{equation}
where $m$ is a nonnegative integer, $0\leq t<1$ such that $m=s+t$ and  the
constant $C$ is independent of $\lambda$ and $f.$
\end{lemma}
\begin{proof} 
Assume that $P$ maps $W^{p,s}(\D_{\alpha\beta})$ onto itself continuously 
and let $T_{\lambda}f=f\circ \tau_{\lambda}.$ Then one can check that 
\[\left\| \left(\frac{\partial}{\partial z}\right)^{P}
\left(\frac{\partial}{\partial \overline  z}\right)^{Q}T_{\lambda}f
\right\|_{L^{p}(\D_{\alpha\beta})}=2^{p_{1}+q_{1}-2/p}\lambda^{2p_{1}
+2q_{1}+|P'|+|Q'|-2n/p} \left\| \left(\frac{\partial}{\partial z}\right)^{P}
\left(\frac{\partial}{\partial \overline  z}\right)^{Q}f
\right\|_{L^{p}(D_{\lambda})} \]
where
$P=(p_{1},\ldots,p_{n}),Q=(q_{1},\ldots,q_{n}),P'=(p_{2},\ldots,p_{n-1}),Q'=(q_{
2},\ldots,q_{n-1}),|P'|=p_{1}+\cdots+p_{n-1},$ and $|Q'|=q_{1}+\cdots+q_{n-1}.$
Therefore we have 
\[\| T_{\lambda}f \|_{W^{p,k}(\D_{\alpha\beta})}\leq
2^{k-2/p}\lambda^{2k-2n/p}\| f \|_{W^{p,k}(D_{\lambda})} .\]
By interpolation we also have 
$\| T_{\lambda}f \|_{W^{p,s}(\D_{\alpha\beta})}\leq
2^{s-2/p}\lambda^{2s-2n/p}\| f \|_{W^{p,s}(D_{\lambda})}$ for all $s>0$.

Let $s=m-t$ where $m$ is a nonnegative integer and $0\leq t<1$. We have 
\begin{equation}\label{HoloSob}
\left\| |r|^{t} \left(\frac{\partial}{\partial z_{1}}\right)^{m}
f\right\|_{L^{p}(\D_{\alpha\beta})} \le  C_1
\left\|f\right\|_{W^{p,s}(\D_{\alpha\beta})}
\end{equation}
for $f$ holomorphic on $\D_{\alpha\beta}$ (see, for example,
\cite{LIgocka87}).

Let $P_{\lambda}$ be the Bergman projection for $D_{\lambda}$. Then
$P_{\lambda}=T_{\lambda}^{-1}PT_{\lambda}$ and 
\begin{eqnarray*}
\left\| |r_{\lambda}|^{t} \left(\frac{\partial}{\partial z_1}\right)^m
P_{\lambda}f\right\|_{L^{p}(D_{\lambda})}
 &=& \left\| |r_{\lambda}|^{t} \left(\frac{\partial}{\partial z_1}\right)^{m}
T_{\lambda}^{-1}PT_{\lambda}f\right\|_{L^{p}(D_{\lambda})}\\
&=&2^{2/p-m}\lambda^{2t+2n/p-2m}\left\| |r|^{t} \left(\frac{\partial}{\partial
z_{1}}\right)^{m} PT_{\lambda}f\right\|_{L^{p}(\D_{\alpha\beta})}\\
&\leq& C_2 \lambda^{2t+2n/p-2m}\left\|
PT_{\lambda}f\right\|_{W^{p,s}(\D_{\alpha\beta})}\\
&\leq& C_3  \lambda^{2n/p-2s}\left\|
T_{\lambda}f\right\|_{W^{p,s}(\D_{\alpha\beta})}\\
&\leq& C_4 \left\| f\right\|_{W^{p,s}(D_{\lambda})}
\end{eqnarray*}
where the constants are independent of $\lambda$.  
\end{proof}
\begin{lemma} \label{SupposeContrary}
If the estimate \eqref{Estimate} holds on $D_{\lambda}$ then 
$$\left\| |r_{\infty} |^{t} \left(\frac{\partial}{\partial z_1}\right)^m
P_{\infty}f\right\|_{L^{p}(D)} \leq C  \left\| f\right\|_{W^{p,s}(D)} $$
where $P_{\infty}$ is the Bergman projection on $D$ and the constant $C$ is
independent of $f.$  
\end{lemma}
The above lemma can be proved like Lemma 1 in \cite{Barrett92}.
\begin{lemma} \label{Bad-f}
Let $s\geq \nu_{\alpha\beta}+n\left(\frac{1}{p}-\frac{1}{2}\right)$ where
$\nu_{\alpha\beta}=\frac{\pi}{ 2\alpha\ln\beta}$ and $s=m-t$ as above. Then
there exists $f\in C^{\infty}_{0}(D)$ such that $|r_{\infty} |^{t}
\left(\frac{\partial}{\partial z_1}\right)^m P_{\infty}f$ is not in $L^p(D)$.
\end{lemma}
\begin{proof} Since $P_{Jk}$ maps $W^{p,\delta}(D)\cap A^p(D)$ onto 
$W^{p,\delta}(D)\cap A^p_{Jk}(D)$ for all $\delta\geq 0$ it is sufficient to
prove that there exists $f\in C^{\infty}_0(D)$ such that
$P_{Jk}P_{\infty}f\not\in W^{p,s}(D).$ Fix $w\in D, J=0,$ and $k=-2$. Let $f$ be
a nonnegative smooth function with compact support in $D$ such that it depends
on $|z-w|$ and $\int_{D}f=1.$ Then $K_{0,-2}(\cdot,w)=P_{0,-2}P_{\infty}f.$ We
can write $s=m-t$ where $m$  is a nonnegative integer and $0\leq t<1.$ In view
of \eqref{HoloSob} above (adapted to $D$) it suffices to show that
$|r_{\infty}(z)|^{t}\frac{\partial^m}{\partial z_1^m}K_{0,-2}(z,w)\not\in
L^p(D)$ for fixed $w.$  Proposition \ref{Kernel} implies that     
\[\frac{\partial^m}{\partial z_1^m}K_{0,-2}(z,w)= 
C z_{1}^{\nu_{\alpha\beta}-n/2-i\mu_{\alpha} -m}
+O\left(z_1^{\nu_{\alpha\beta}-n/2+\varepsilon-m}
\right).\]
Let 
\begin{multline*}
D'= \Big\{(z_{1},z',z_{n})\in \C^{n}:\re\left( z_{1}e^{-2\alpha
i\ln|z_{n}|}\right) >|z'|^{2},1+\delta<|z_{n}|<\beta-\delta,\\
|z_1|<\delta, \Big|\theta_1-2\alpha \ln|z_n|\Big|<\frac{\pi}{4}\Big\}
\end{multline*}
for suitably small $\delta>0$.
Then $|r_{\infty}|$ is comparable to $|z_1|$ on $D'$ and 
\begin{align*} 
\int\limits_{D}|r_{\infty}(z)|^{pt}\left|\frac{\partial^m}{\partial
z_1^m}K_{0,-2}(z,w)\right|^{p}dV(z)  
&\ge  \int\limits_{D'}|r_{\infty}(z)|^{pt}\left|\frac{\partial^m}{\partial
z_1^m}K_{0,-2}(z,w)\right|^{p}dV(z)\\
 &\ge c \int_{0}^{\delta} r_1^{p\nu_{\alpha\beta}+pt-pm+n-1-pn/2} dr_1
\end{align*}
where $c$ is a positive constant. The last integral above is divergent if $s\geq
\nu_{\alpha\beta}+n\left(\frac{1}{p}-\frac{1}{2}\right).$ Therefore   
\[|r_{\infty}(z)|^t\frac{\partial^m}{\partial z_1^m}P_{0,-2}P_\infty f
=|r_{\infty}(z)|^t\frac{\partial^m}{\partial z_1^m}K_{0,-2}(z,w)\not\in
L^p(D)\] 
for $s\geq \nu_{\alpha\beta}+n\left(\frac{1}{p}-\frac{1}{2}\right).$
\end{proof}

\section{Acknowledgment} 
We would like to  thank the referee for pointing out a mistake  in an
earlier version of this manuscript.


\end{document}